\documentclass[12pt,reqno]{amsart}

\usepackage{graphicx} 
\usepackage{eucal}
\usepackage{amssymb,url,amscd,amsfonts}
\usepackage{latexsym}
\usepackage{hyperref}
\usepackage[all]{xy}
\theoremstyle{plain}
\newtheorem{theorem}{Theorem}

\newtheorem{corollary}[theorem]{Corollary}
\newtheorem{lemma}[theorem]{Lemma}
\newtheorem{proposition}[theorem]{Proposition}

\numberwithin{equation}{section} 
\numberwithin{theorem}{section} 

\newtheorem{example}{Example}

\begin{document}
	
	\title[CMC Radial graphs over domains of $\mathbb S^n$]
	{Constant mean curvature Radial \\ graphs over domains of $\mathbb{S}^n$ }
	\author{Flávio Cruz, José T. Cruz and  Jocel Oliveira}
	\address{Departamento de Matem\'atica\\ 
		Universidade Regional do Cariri\\ 
		Campus Crajubar \\ 
		Juazeiro do Norte, Cear\'a -  Brazil\\ 63041-145}
	\email{flavio.franca@urca.br}
		\address{Departamento de Matem\'atica\\ 
		Universidade Regional do Cariri\\ 
		Campus Crajubar \\ 
		Juazeiro do Norte, Cear\'a -  Brazil\\ 63041-145}
	\email{tiago.cruz@urca.br}
		\address{Departamento de Matem\'atica\\ 
		Universidade Regional do Cariri\\ 
		Campus Crajubar \\ 
		Juazeiro do Norte, Cear\'a -  Brazil\\ 63041-145}
	\email{jocel.faustino@urca.br}
	
	\thanks{Research partially supported by FUNCAP Grant No. BP6-00241-00347.01.00/25}
	
	\begin{abstract} 
	We establish the existence of hypersurfaces with constant mean curvature and a prescribed boundary in Euclidean space, represented as radial graphs over domains of the unit sphere. Under the assumptions that the mean curvature of the domain's boundary is positive and that a subsolution exists for the associated Dirichlet problem, we extend Serrin's classical result to include the case of positive constant mean curvature.

	\end{abstract}
	
	\maketitle
	
	\section{Introduction}
	
	The existence of hypersurfaces with constant mean curvature (in short, CMC hypersurfaces) and a prescribed boundary is a research topic that has attracted the attention of many mathematicians in recent decades.  A useful strategy to obtain a CMC hypersurfaces with a prescribed boundary in Euclidean space is to describe them nonparametrically as solutions of a Dirichlet problem for a certain quasilinear elliptic PDE. Hence, the solutions are then graphs over domains in umbilical hypersurfaces of Euclidean space. The classic references on the subject are \cite{TRU}, \cite{JS}, and \cite{SER}. Here, we are interested in the case that the umbilical hypersurface is a sphere.
	
	We now explain more precisely the framework we are considering. Let $\Omega\subset \mathbb{S}^n$ be a smooth domain (that is, open and connected) of the unit sphere $\mathbb{S}^n\subset \mathbb{R}^{n+1}$ in the $(n+1)$-dimensional Euclidean space. The \textit{radial graph} over $\Omega$ of a function $u\in C^2(\Omega)\cap C^0(\overline{\Omega})$ is the hypersurface
	\[
	\Sigma=\Sigma(u)=\{X(p)=e^{u(p)}p:\;p\in\overline{\Omega}\}\subset\mathbb{R}^{n+1},
	\]
	where $X$ denotes the position vector of $\Sigma$ in $\mathbb{R}^{n+1}$. A radial graph $\Sigma$ is characterized as a star-shaped surface with respect to the origin, that is, for each half-line starting from the origin, it intersects $\Sigma$ at most at one point. In particular, if $N$ is the unit normal vector of $\Sigma$, then the support function $\langle N(X),X\rangle$ is either positive or negative everywhere.
	In this work, we assume that the orientation of $\Sigma$ is the natural one, that is, $\langle N(X),X\rangle<0$, for all $X\in \Sigma$.
	
	In this work, we consider the problem of finding radial graphs with constant mean curvature and a prescribed boundary. After notable contributions by Radó \cite{Rad} and Tausch \cite{TAU}, focused on minimal surfaces, Serrin \cite{SER} established the existence of radial graphs with constant \textit{non-positive} mean curvature under a suitable smallness condition involving the boundary datum and the geodesic mean curvature of the boundary  $\partial\Omega$ with respect to $\mathbb S^n$. Precisely, Serrin proved the following theorem.
	\begin{theorem}\label{serrin}(Serrin, \cite{SER}) Let $\Omega$ be a smooth domain of $S^n$ whose closure is contained in an open hemisphere. Denote by $\mathcal{H}_{\partial\Omega}$ the mean curvature of $\partial \Omega$ as a submanifold of $\Omega$,  computed with respect to the unit normal pointing to the interior of $\Omega$. Let $\varphi$ be a smooth function on $\partial \Omega$ and $H$ be a non-positive smooth function defined on $\overline{\Omega}$ and such that
		\begin{equation}
			\mathcal{H}_{\partial\Omega}(p)\geq -\frac{n}{n-1}H(p)e^{\varphi(p)}\geq 0
		\end{equation}
		for each $p\in \partial \Omega$. Then there exists a unique radial graph $\Sigma$ on $\Omega$ with mean curvature $H$ and boundary data $\varphi$.
	\end{theorem}
	We emphasize that, in Serrin’s theorem, the mean curvature of the boundary of the domain is necessarily non-negative $\mathcal{H}_{\partial\Omega} \geq 0$. In particular, for $n=2$, this condition implies that the domain is contained in a hemisphere of $\mathbb S^2$. Therefore, this last condition can be excluded as a hypothesis when we deal with surfaces in the Euclidean three-dimensional space.
	
	Several versions, improvements, and variations of Serrin's Theorem have been developed in different ways, including in recent years (see, e.g., \cite{A-D}, \cite{Cald}, \cite{Rip}, \cite{Lop},  \cite{Nelli}, and references therein). However, to our knowledge, all results focus on cases with non-positive mean curvature $H$, except for \cite{Cald}, which deals with prescribed \textit{non-constant} positive functions $H$ defined in $\mathbb{R}^{n+1}$. Although the authors do not explain how they derive the key boundary gradient estimate in \cite{Cald}. The main difficulty in handling the positive mean curvature case (with the orientation fixed above) lies in the fact that the mean curvature equation does not satisfy the maximum principle in these conditions.
	Here, we deal with the problem of finding radial graphs with \textit{positive} constant mean curvature in domains that are not necessarily contained in a hemisphere of $\mathbb S^n$. Specifically, we prove the following result.
	
	\begin{theorem}\label{main}
		Let $\Omega \subset \mathbb{S}^n$ be a smooth domain and let $\varphi$ be a smooth function on $\partial \Omega$. 
		Assume the mean curvature of the boundary $\partial \Omega$ as a submanifold of $\Omega$, computed with respect to the unit normal pointing to the interior of $\Omega$,
		is positive and that there exists a radial graph $\overline{\Sigma}$ with mean curvature $H_{\overline{\Sigma}}>0$ and boundary data $\varphi$. Then, for any constant $H$ satisfying $0<H<H_{\overline{\Sigma}}$, there exists a radial graph $\Sigma$ with constant mean curvature $H$ and boundary data $\varphi$.
	\end{theorem}
	
	We observe that the hypothesis concerning the existence of a subsolution was also introduced in the context of the related existence problem for the Gauss curvature, as explored in \cite{G-S} (see also \cite{FF} and references therein). Uniqueness does not hold in Theorem \ref{main}, as we can see by the Example \ref{example} in Section \ref{Preliminaries}. 
	
	Using the domain $\Omega$ as a subsolution, we derive the following corollary from Theorem \ref{main}.
	\begin{corollary}
		Let $\Omega\subset S^n$ be as in Theorem \ref{main}. If $0<H\leq 1$ then there exists a radial graph $\Sigma$ with constant mean curvature $H$ and boundary $\partial \Sigma=\partial \Omega$. 
	\end{corollary}
	
	The proof of Theorem \ref{main} employs the continuity method and degree theory, building on a priori estimates. The paper is organized as follows: Section \ref{Preliminaries} introduces notation and the associate Dirichlet problem. Section \ref{Estimates} establishes the a priori height and boundary gradient estimates, while the gradient interior estimate is developed in Section \ref{interior}.  The existence proof is presented in Section \ref{Existence}.
	
\subsection* {Acknowledgements}

The authors would like to thank Leandro Pessoa for his interest in this work and for various helpful conversations.
Part of this work was completed while José T. Cruz was a guest at the Departamento de Matemática of Universidade Federal do Piauí, whose hospitality is gratefully acknowledged. Flávio F. Cruz thanks FUNCAP for financial support through the BPI program of Grant No. BP6-00241-00347.01.00/25.

	\section{Preliminaries}
	\label{Preliminaries}
	
	We first show that proving Theorem \ref{main} depends on finding a solution to the Dirichlet problem for a quasilinear PDE. We start by reviewing the second fundamental form and other key geometric quantities for a smooth radial graph $\Sigma$ defined by $X(x) = e^{u(x)}x$, where $u$ is a smooth function on a domain $\Omega$ in the unit sphere $\mathbb{S}^n \subset \mathbb{R}^{n+1}$.
	
	Let $e_1,\ldots, e_n$ be a smooth local orthonormal frame on $\mathbb S^n$ and let $\nabla$ denote the covariant derivative on $S^n$. The metric of $\Sigma$ in terms of $u$ is given by
	\begin{equation}
		g_{ij}=\langle \nabla_iX,\nabla_jX\rangle=e^{2u}\left(\delta_{ij}+\nabla_iu\nabla_ju\right)
	\end{equation}
	where $\nabla_i=\nabla_{e_{i}}$ and $\langle\cdot,\cdot \rangle$ denotes  the standard inner product in $\mathbb{R}^{n+1}$. The unit normal of $\Sigma$ on $X$ is given by
	\begin{equation}
		N=\frac{1}{\sqrt{1+|\nabla u|^2}}\left(\nabla u-X\right)
	\end{equation}
	where $\nabla u =\mbox{grad } u$, and the second fundamental form of $\Sigma$ is
	\begin{equation}
		h_{ij}=\langle \nabla_{ij} X,N\rangle = \frac{e^u}{\sqrt{1+|\nabla u|^2}}\left(\delta_{ij}+2\nabla_iu\nabla_ju-\nabla_{ij}u\right)
	\end{equation}
	where $\nabla_{ij}=\nabla_i\nabla_j$. The mean curvature $H$ of $\Sigma$ is given by 
	\[
	nH=\sum_{i,j}g^{ij}h_{ij},
	\]
	where $(g^{ij})$ is the inverse of $(g_{ij})$.
	A straightforward computation shows that $H$ satisfies the equation
	\begin{equation}\label{12}
		Q[u]=\frac{1}{e^uW}\left(\sum_{i,j}\left(\delta_{ij}-\frac{\nabla_iu\nabla_ju}{W^2}\right)\nabla_{ij}u-n\right)=-nH
	\end{equation}
	where $W^2=1+|\nabla u|^2$.  Equation \eqref{12} is a quasilinear elliptic equation in $\Omega\subset\mathbb S^n$ and the machinery of Schauder theory can be used in the problem of existence. Thus, Theorem \ref{main} is equivalent to find a solution $u\in C^2(\Omega)\cap C^0(\overline\Omega)$ of the Dirichlet problem
	\begin{equation}\label{pd}
		\left\{\begin{array}{cl}
			Q[u]=-nH & \mbox{ in } \Omega \\
			u=\varphi & \mbox{ on } \partial \Omega.
		\end{array}\right.
	\end{equation}
	The operator $Q$ is strictly positive elliptic, but the zero order term in its linearized form is positive and not easily controlled.
	As a result, when using the continuity method and degree theory to show a solution exists, we find that equation \eqref{pd} is not suitable to work with. Therefore, in order to prove the existence of solutions for \eqref{pd}, we will use two auxiliary forms of \eqref{pd}. Precisely, we will work on the Dirichlet problem
	\begin{equation}\label{pd2}
		\left\{\begin{array}{cl}
			Q[u]=\Upsilon(X) & \mbox{ in } \Omega \\
			u=\varphi & \mbox{ on } \partial \Omega,
		\end{array}\right.
	\end{equation}
	where $\Upsilon$ is a non-positive smooth function defined in 
	\[
	\Delta=\left\{X\in \mathbb{R}^{n+1} : \frac{X}{|X|}\in \Omega\right\}
	\]
	and satisfies
	\begin{equation}
		\label{decres-rho}
		\frac{\partial }{\partial \rho}(\rho \Upsilon (\rho x))\geq 0
	\end{equation}
	for $\rho \leq \overline{\rho}$, where $\overline{\rho}$ is defined by $\overline\Sigma=\{\overline \rho(x) x : x\in \overline \Omega\}$. The two specific forms of $\Upsilon$ will be presented in Section \ref{Existence}.  Condition \eqref{decres-rho} is crucial for deriving the interior gradient estimate (see \cite{CSN-IV}).
	
	Next, we provide an example showing that uniqueness does not hold in Theorem \ref{main}. 
	\begin{example}
		\label{example}
		Consider a sphere in $\mathbb{R}^3$ given by
		\[
		\mathbb{S}(h)=\{(x,y,z)\in\mathbb{R}^3; x^2+y^2+(z+h)^2=R_h^2\},
		\]
		where $1/\sqrt{2}\leq r<1$, $0\leq h\leq \frac{2r^2-1}{\sqrt{1-r^2}}$ and $R_h=\sqrt{1+h^2+2h\sqrt{1-r^2}}$. We denote $\mathbb{S}(0)=\mathbb{S}^2$. Let $M_1=M_1(h)$ be the part of $\mathbb{S}(h)$ above the plane as described in Figure \ref{fig2} below, given by $z=\sqrt{1-r^2}$. Now, let $M_2=M_2(h)$ be the reflection of $\mathbb{S}(h)- M_1$ with respect to the plane. 
		\begin{center}
			\begin{figure}[h]
				\includegraphics[scale=0.7]{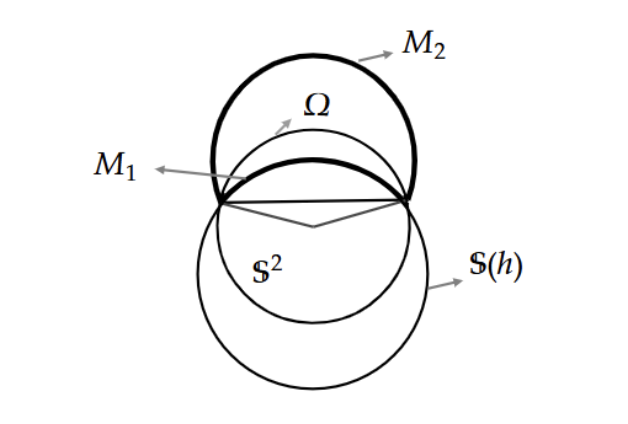}
				\caption{Non-uniqueness of radial graphs for $H>0$.}
				\label{fig2}
			\end{figure}
		\end{center}
		Under these conditions, $M_1$ and $M_2$ are radial graphs over the domain
		\[
		\Omega=\{(x,y,z)\in\mathbb{R}^3; z>\sqrt{1-r^2}\}\subset\mathbb{S}^2,
		\]
		and both surfaces have the same mean curvature $H_h=\frac{1}{R_h}>0$. Moreover, we have $0\leq H_h\leq 1$.
	\end{example}
	
	\section{Height and Boundary Gradient Estimates}
	\label{Estimates}
	In this section, we establish a priori estimates for the height and gradient along the boundary of the solutions of \eqref{pd2}.  First, we deduce some well-known properties of the distance function $d=\textrm{dist}(\cdot, \partial\Omega)$ from $\partial\Omega$. We denote by $\Omega_\epsilon=\{x\in\Omega : d(x)<\epsilon\}$ and $\Gamma_\epsilon=\{x\in\Omega : d(x)=\epsilon\}$. Since $\partial\Omega$ is smooth, $d$ is smooth in a small tubular neighborhood $\Omega_\epsilon$ of $\partial\Omega$. We may choose a smooth local orthonormal frame $e_1,\ldots, e_n$ on $\Omega_\epsilon$ such that $e_n=\nabla d$ is the unit inward normal vector along $\Gamma_d$, for $0\leq d\leq\epsilon$. It follows from $|\nabla d|=1$ that 
	\begin{equation}
		\label{d1}
		\sum_i\nabla_i d\nabla_{ij} d=0
	\end{equation}
	for any $0\leq j\leq n-1$. We also have
	\begin{equation}
		\label{d2}
		\Delta d|_{\Gamma_\epsilon}= -(n-1)H|_{\Gamma_\epsilon},
	\end{equation}
	where $H|_{\Gamma_\epsilon}$ denotes the mean curvature of $\Gamma_\epsilon$ with respect to the normal $e_n=\nabla d$. Indeed, \eqref{d2} holds if we replace $\Omega_\epsilon$ with the larger subset $\Omega_0 \subset \Omega$, which includes points connected to $\partial\Omega$ by a unique minimizing geodesic. It was demonstrated in \cite{YY-NIR} that the regularity of the function $d$ in $\Omega_0$ matches that of $\partial\Omega$. Using this observation, we follow \cite{DL} and prove the following result.
	\begin{lemma}
		\label{lemma-d}
		Let $y_0\in\partial\Omega$ be the closest point to a given point $x_0\in\Gamma_\epsilon\subset\Omega_0$. Then,
		\begin{equation}
			\label{d3}
			H|_{\Gamma_\epsilon}(x_0)\geq \mathcal H_{\partial\Omega}(y_0).
		\end{equation}
	\end{lemma}
	\begin{proof}
		We compute on $\Gamma_\epsilon$, for $1\leq i,j\leq n-1$,
		\begin{equation}
			\label{d4}
			-\frac{d}{d\epsilon}\langle A_\epsilon e_i, e_j\rangle = e_n\langle\nabla_i e_n, e_j\rangle =\langle A_\epsilon e_i, A_\epsilon e_j\rangle -\langle R(e_i, e_n)e_n,e_j\rangle
		\end{equation}
		where $A_\epsilon$ is the Weingarten operator of $\Gamma_\epsilon$ and $R$ is the curvature tensor in $\mathbb{S}^n$. On the other hand,
		\begin{align}
			\begin{split}
				\frac{d}{d\epsilon}\langle A_\epsilon e_i, e_j\rangle &= \langle\nabla_n A_\epsilon e_i, e_j\rangle+\langle A_\epsilon e_i, \nabla_n e_j\rangle \\ &= \langle A'_\epsilon e_i, e_j\rangle-2\langle A_\epsilon e_i, A_\epsilon e_j\rangle.
			\end{split}
		\end{align}
		Adding these equations, we obtain the well-known Ricatti equation
		\begin{equation}
			\label{ricatti}
			A'_\epsilon -A_\epsilon^2-R_\epsilon=0,
		\end{equation}
		where $R_\epsilon= R(\cdot, e_n)e_n|_{d=\epsilon}$. Thus, taking traces,
		\begin{equation}
			n\dfrac{d}{d\epsilon}H|_{\Gamma_{\epsilon}}=\textrm{tr}(A_\epsilon^2+R_\epsilon)>0.
		\end{equation}
		It follows that $H|_{\Gamma_{\epsilon}}$ does not decrease with increasing $\epsilon$.
	\end{proof}
	We are now ready to establish the a priori height estimate.
	\begin{proposition}
		Let $u\in C^2(\Omega) \cap C^0(\overline{\Omega})$ be a solution of \eqref{pd2}
		satisfying $u\leq \overline{u}$. Then, there exists a uniform constant  $L = L(n, \Omega, \overline{u})$ such that  
		\begin{equation}
			\label{height}
			|u| \leq L.
		\end{equation}
	\end{proposition}
	\begin{proof}
		We construct lower-barrier functions for $u$ in \eqref{pd2} of the form
		\begin{equation}
			w = f\circ d + \inf_{\partial\Omega}\varphi 
		\end{equation}
		where $f$ is a real function that will be chosen later. First, we observe that a solution $u$ of \eqref{pd2} satisfies
		\begin{equation}
			\label{H1}
			\tilde Q[u]:=\sum_{i,j}\dfrac{1}{W}\left(\delta_{ij}-\frac{\nabla_iu\nabla_ju}{W^2}\right)\nabla_{ij}u-\dfrac{n}{W}=e^u\Upsilon\leq 0.
		\end{equation}
		Along $\Omega_0$ we have
		\begin{align*}
			\tilde Q[w]& =  \sum_{i,j}\dfrac{1}{\sqrt{1+(f')^2}}\left(\delta_{ij}-\frac{(f')^2\nabla_i d \nabla_j d}{1+(f')^2}\right)(f''\nabla_i d\nabla_jd+f'\nabla_{ij} d)\\ &-\dfrac{n}{\sqrt{1+(f')^2}}\\
			&= \dfrac{1}{\sqrt{1+(f')^2}}\left(f' \Delta d+ f''-n\right)-\frac{(f')^2f''}{(1+(f')^2)^{3/2}}.
		\end{align*}
		We choose
		\begin{equation}
			\label{H2}
			f(d)=-\frac{e^{CA}}{C}(1-e^{-Cd}),
		\end{equation}
		where the constant $C>0$ will be chosen later and $A$ is such that $A>\textrm{diam}(\Omega)$. Then, as $f''=-Cf'$, we obtain 
		\begin{equation}
			\label{H3}
			\tilde Q[w]= \dfrac{f'\Delta d-n}{\sqrt{1+(f')^2}} -\dfrac{Cf'}{(1+(f')^2)^{3/2}}.
		\end{equation}
		It follows from\eqref{d2} and Lemma \ref{lemma-d} that
		\begin{equation}
			\label{H4}
			\Delta d|_{\Gamma_d}= -(n-1)H|_{\Gamma_d}\leq -(n-1) \mathcal{H}_{\partial \Omega}<0,
		\end{equation}
		since $\mathcal{H}_{\partial \Omega}>0$. Thus, as $f'(d)= -e^{C(A-d)}<0$, we have
		\begin{equation}
			\label{H5}
			\tilde Q[w]\geq -\dfrac{(n-1)f'\mathcal{H}_{\partial \Omega}+n}{\sqrt{1+(f')^2}} -\dfrac{Cf'}{(1+(f')^2)^{3/2}}.
		\end{equation}
		As $f'\rightarrow -\infty$ as $C\rightarrow +\infty$, choosing $C=C(n,\Omega)$ sufficiently large, we obtain $\tilde Q[w]> 0$ since $\mathcal{H}_{\partial \Omega}>0$. We conclude that at points of $\Omega_0$ it holds
		\begin{equation}\label{18}
			\left\{\begin{array}{cl}
				\tilde Q[w]>\tilde Q[u]& \textrm{in } \Omega_0 \\
				w\leq u & \textrm{on } \partial \Omega.
			\end{array}\right.    
		\end{equation}
		Let us prove that $w\leq u$ in $\Omega$. We adopt the reasoning outlined in \cite{DL}. Suppose, by contradiction, that there are points where the function $\hat u:=w-u$ satisfies $\hat u>0$. Then, $m:=\hat u(y)>0$ at a point of maximum $y\in \overline\Omega$ and $\hat u$. Let $\gamma$ be a minimizing geodesic that connects $y$ to $\partial \Omega$ and such that $d=\textrm{dist}(y,\partial \Omega)$ is achieved. So, we can write $\gamma(t)=\exp_{y_0}(t\eta)$, $0\leq t\leq d$, where $y_0\in \partial \Omega$ and $\eta$ is a unit vector in $T_{y_0}\mathbb S^n$. Since $\gamma$ is minimizing, we have $\textrm{dist}(\gamma(t),\partial\Omega)=t$ and the function $w$ restricted to $\gamma$ is differentiable with $w'(\gamma(t))=-e^{C(A-t)}$. Since the maximum of $\hat u$ restricted to $\gamma$ occurs at $t=d$, at the point $y$, then
		\begin{equation}
			\hat u '(\gamma(d))=w'(\gamma(d))-u'(\gamma(d))\geq 0.
		\end{equation}
		Hence
		\begin{equation}
			\langle \nabla u(y),\gamma '(d) \rangle =u'(\gamma(d))\leq w'(\gamma(d))=-e^C(A-d)<0.
		\end{equation}
		In particular, $\nabla u(y)\neq0$ and therefore, the level hypersurface 
		\begin{equation*}
			\mathcal{L}=\{x\in \Omega \cap B_r(y)~|~u(x)=u(y)\}
		\end{equation*}
		is regular for sufficiently small radius $r$. Along $\mathcal{L}$ we have
		\begin{equation}
			\hat u(x)-w(x) = \hat u(y)-w(y)\geq \hat u(x)-w(y).
		\end{equation}
		Thus
		\begin{equation}\label{eq19}
			w(x)\leq w(y), \quad \forall x\in \mathcal{L}.
		\end{equation}
		Since $w$ is decreasing in $d=\textrm{dist}(\cdot, \partial \Omega)$,  \eqref{eq19} ensures that $d(x)\geq d(y)=d$, for any $x\in \mathcal{L}$. Thus, the points in $\mathcal{L}$ are at a distance at least $d$ from $\partial \Omega$, i.e.,  $d(x)\geq d$,
		for any $x\in\mathcal L$. Since $\mathcal{L}$ is of class $C^2$ it satisfies the interior sphere condition: there exists a small ball $B_\epsilon(z)$ touching $\mathcal{L}$ at $y\in \mathcal{L}$ and such that $B_\epsilon (y)$ is on the side for which $-\nabla u(y)$ and $\gamma' (d)$ point. Therefore, the points $x\in B_\epsilon (z) $ satisfy $u(x)\geq u(y)$ and hence
		\begin{equation*}
			m-w(x)=\hat u(y)-w(x) \geq -u(y)= \hat u(y)-w(y)=m-w(y),
		\end{equation*}
		for any $x\in B_\epsilon(z)$, where we use the definition of $m$. Again, as $w$ is decreasing in $d$, we conclude that
		\begin{equation}
			\label{20}
			d(x)=\textrm{dist}(x,\partial \Omega)\leq \textrm{dist}(y,\partial \Omega)=d(y)=d
		\end{equation}
		for $x\in B_\epsilon(z)$ and consequently $B_\epsilon (z)\subset \Omega$ and it is far away from $\partial \Omega$. Therefore, we can extend the geodesic $\gamma$ through $B_\epsilon(z)$. We claim that the center $z$ of $B_\epsilon(z)$ belongs to the this extension of the geodesic $\gamma$. In fact, if the extension of $\gamma$ did not pass through $z$, then the curve constructed by joining the segment of $\gamma$ from $y_0$ to $y \in \partial B_\epsilon(z)$ would have a length smaller than that of a minimizing geodesic joining $z$ to $y_0 \in \partial\Omega$, which would lead to a contradiction. Thus, if there exists at least two distinct minimizing geodesics joining $y$ to the boundary $\partial \Omega$, then the point $z$ belongs to the extension of both geodesics after its intersection that occurs at $y$. Choosing $\epsilon >0$ small enough, we conclude that this configuration cannot occur. Therefore there exists a unique minimizing geodesic that realizes the distance $d=\textrm{dist}(y,\partial \Omega)$ and so $y\in \Omega_0$ However, in this case, $\hat u(y)\leq 0$, a contradiction, by our previous computation. We conclude that $w\leq u$ throughout $\overline{\Omega}$ and hence $w$ is a continuous (viscosity) subsolution for the Dirichlet problem 
		\begin{equation}
			\label{pd-tilde}
			\left\{\begin{array}{cl}
				\tilde Q[u]= 0 & \textrm{in } \Omega_0 \\
				u = \varphi & \textrm{on } \partial \Omega.
			\end{array}\right.    
		\end{equation}
		Then, any solution of \eqref{pd2} satisfies $w\leq u\leq \overline{u}$, so  \eqref{height} is proved.
	\end{proof}
	
	We now address the boundary gradient estimate. Initially, we extend $\varphi$ to a tubular neighborhood $\Omega_\epsilon$ of $\partial\Omega$ by defining $\varphi(q) = \varphi(p)$ for $q = \exp_p (s e_n)$, where $p \in \partial\Omega$ and $e_n = \nabla d$ is the unit normal pointing toward the interior of $\Omega$. We construct barriers for $u$ in \eqref{pd2} of the form
	\begin{equation}
		w = \varphi + h \circ d
	\end{equation}
	defined in a tubular neighborhood $\Omega_\epsilon$ of $\Omega$, for some real function $h$ to be chosen later and $d=\textrm{dist}(\cdot, \partial\Omega)$. Let $\tilde Q$ be the operator defined in \eqref{H1}. Thus,
	\begin{equation}
		\label{BD1}
		\tilde{Q}[w]=\dfrac{\Delta w-n}{\sqrt{1+|\nabla w|^2}}-\frac{\nabla_i w\nabla_jw\nabla_{ij}w}{(1+|\nabla w|^2)^{3/2}}.
	\end{equation}
	We have
	\begin{equation}
		\nabla w = \nabla \varphi+f'\nabla d \quad \textrm{and} \quad   \Delta w = \Delta\varphi+f''+f'\Delta d.
	\end{equation}   
	In particular, $|\nabla w|^2 =|\nabla \varphi|^2+(f')^2$. We now compute
	\begin{align}
		\begin{split}
			\label{BD2}
			\nabla_iw  \nabla_jw & \nabla_{ij}w = \nabla_i\varphi \nabla_j \varphi \nabla_{ij}\varphi + f' \nabla_i\varphi \nabla_j\varphi \nabla_{ij}d \\  + & 2f' \nabla_i\varphi\nabla_j d \nabla_{ij}\varphi   +  (f') ^2 \nabla_id\nabla_j d \nabla_{ij}\varphi + (f')^2 f''
		\end{split}
	\end{align}     
	where we used $\langle \nabla d,\nabla \varphi\rangle =0$ and $|\nabla d|=1$. We define
	\begin{equation}
		\label{BD3}
		f(d)=-\mu\ln{(1+Kd)}
	\end{equation}
	for certain positive constants $\mu$ and $K$ to be chosen later. Notice that
	\begin{equation}
		\label{BD4}
		f'(d)=\frac{-\mu K}{1+\mu d}\quad \textrm{and}\quad f''=\frac{1}{\mu}(f')^2.
	\end{equation}
	Replacing \eqref{BD2} and \eqref{BD4} into the expression for $\tilde Q[w]$ in \eqref{BD1}, we obtain
	\begin{align}
		\begin{split}
			\label{BD5}
			\tilde{Q}&[w] 
			= \dfrac{f'\Delta d+\Delta \varphi-n}{\sqrt{1+|\nabla\varphi|^2+(f')^2}} - \dfrac{\mu^{-1}(f')^2(1+|\nabla \varphi|^2)}{(1+|\nabla \varphi|^2+(f')^2)^{3/2}}\\ &  - \dfrac{1}{(1+|\nabla \varphi|^2+(f')^2)^{3/2}}\Big(\nabla_i\varphi\nabla_j\varphi\nabla_{ij}\varphi+f'\nabla_i\varphi\nabla_j\varphi\nabla_{ij}d  \\ &  +2f'\nabla_i\varphi\nabla_jd\nabla_{ij}\varphi+(f')^2\nabla_id\nabla_jd\nabla_{ij}\varphi\Big).
		\end{split}
	\end{align}
	We choose 
	\begin{equation}
		\label{BD6}
		\mu=\frac{C}{\ln{(1+K)}}
	\end{equation}
	for some constant $C>0$ to be chosen. In particular, $\mu\rightarrow 0$ as $K\rightarrow +\infty$. We also have
	\begin{equation*}
		f'(0)=-\dfrac{CK}{\ln{(1+K)}}\rightarrow -\infty
	\end{equation*}
	Therefore, choosing $K$ and then $C$ large enough, we assure that $\tilde Q[w]>0$ on a small tubular neighborhood $\Omega_\epsilon$ of $\partial\Omega$ since the leading term (asymptotically) in \eqref{BD5} is
	\begin{equation*}
		\dfrac{f'\Delta d}{\sqrt{1+|\nabla\varphi|^2+(f')^2}}=-\dfrac{f'(n-1)\mathcal{H}_{\partial\Omega}}{\sqrt{1+|\nabla\varphi|^2+(f')^2}}>0,
	\end{equation*}
	as $\mathcal{H}_{\partial\Omega}>0$. Moreover, from the height estimates it follows that $w\leq u$ in $\partial\Omega$ for uniform large constants $K$ and $C$. Thus, $w$ is a lower barrier for the Dirichlet problem \eqref{pd-tilde}. Then, from $w\leq u\leq \overline{u}$ in $\Omega_\epsilon$ and $w=u=\overline{u}$ on $\partial\Omega$ we obtain the following result.
	\begin{proposition}
		\label{P-BD}
		Let $\Omega$ be as in Theorem \ref{main} and $ u\leq \overline{u}$ be a solution of (\ref{pd2}). Then
		\begin{equation}
			|\nabla u|\leq C\qquad \mbox{on}\quad \partial \Omega,
		\end{equation}
		for a uniform constant $C=C(n,\Omega,\overline{u})$.
	\end{proposition}
	
	\section{Interior Gradient Estimates}
	\label{interior}

	Now, as the last step in proving a priori estimates for solutions $u$ of \eqref{pd2} satisfying $u\leq \overline{u}$, we will establish an interior estimate for the gradient of solutions.
	\begin{proposition}
		Let $\Omega$ be as in Theorem \ref{main} and $u\leq \overline{u}$ be a solution of (\ref{pd2}). Assume that $\Upsilon$ satisfies \eqref{decres-rho}, then there exists a uniform constant $C=C(n,\Omega,\overline{u})$ such that
		\begin{equation}
			\label{IGE}
			|\nabla u|\leq C \quad \textrm{in} \quad \Omega.
		\end{equation}
	\end{proposition}
	\begin{proof}
		Let  $u$ be a solution of  $Q[u]=\Upsilon (X)$, $X=e^ux, x\in\Omega$. 	The usual elliptic regularity allow us to assume that $u\in C^3(\Omega)$. We consider the auxiliary test function
		\[
		\chi=ve^{2Ku}
		\]
		where $v=|\nabla u|^2$ and $K$ is a positive constant to be defined later. If $\chi$ achieves its maximum on $\partial\Omega$ we have a bound for $|\nabla u|$ in $\Omega$ as desired. Hence, we may assume that the maximum is attained at an interior point $x_0\in\Omega$. We choose a local orthonormal frame $e_1, \ldots, e_n$ around $x_0$ such that
		\begin{equation}
			\label{CC1}
			e_1=\frac{1}{|\nabla u|} \nabla u
		\end{equation}
		and $\nabla^2 u$ is diagonal at $x_0$. Thus, at $x_0$, we have
		\begin{equation}
			\label{CC2}
			\nabla_1 u= v^{1/2} \quad\textrm{and} \quad \nabla_j u= 0, j\geq 2.
		\end{equation}
		Moreover, it follows from $\nabla \chi(x_0)=0$ that
		\begin{equation}
			\label{CC3}
			\nabla_{11}u=-Kv \quad\textrm{and}\quad \nabla_j v =\nabla_{ij} u=0, j\geq 2.
		\end{equation}
		Now we rewrite equation $Q[u]=\Upsilon$ as
		\begin{equation}
			\label{CC4}
			((1+v)\delta_{ij}-\nabla_i u\nabla_j u)\nabla_{ij} u-(1+v)n =(1+v)^{3/2}e^u\Upsilon.
		\end{equation}
		In particular, at $x_0$, we have
		\begin{equation}
			\label{CC7}
			(1+v)\sum_{i} \nabla_{ii} u=(1+v)^{3/2}e^u\Upsilon -Kv^2+n(1+v).
		\end{equation}
		Differentiating covariantly we obtain
		\begin{align}
			\label{CC5}
			\begin{split}
				\left((1+v)\delta_{ij} - \nabla_i u\nabla_j u\right)\nabla_{ijk}u= -2\left(\nabla_\ell u\nabla_{\ell k} u \, \delta_{ij} - \nabla_i u\nabla_{jk} u\right)\nabla_{ij} u \\  + 2n\nabla_i u\nabla_{ik}u  +3(1+v)^{1/2}\nabla_iu\nabla_{ik} u e^u\Upsilon  + (1+v)^{3/2} \nabla_k(e^u\Upsilon) .
			\end{split}
		\end{align}
		Contracting with $\nabla u$, we conclude from \eqref{CC2}, \eqref{CC3} and \eqref{CC4} that
		\begin{align}
			\label{CC6}
			\begin{split}
				\left((1+v)\delta_{ij} - \nabla_i u\nabla_j u\right)\nabla_k u\nabla_{ijk}u =& -4Kv^2(1+v)^{1/2}e^u\Upsilon \\ +2K^2v^3  -2K^2v^4(1+v)^{-1} +&(1+v)^{3/2} v^{1/2}\nabla_1 (e^u\Upsilon).
			\end{split}
		\end{align}
		In order to eliminate the third derivative in \eqref{CC6} we use that 
		\[
		\nabla_{ij}\chi(x_0)\leq 0.
		\] 
		We compute
		\begin{align}
			\begin{split}
				\label{CC10}
				\nabla_{ij} \chi = &e^{2Ku} \Big( 4K(\nabla_j u \nabla_k u \nabla_{ki} u + Kv\nabla_i u \nabla_j u +\nabla_iu\nabla_ku\nabla_{kj}u)  \\ & + 2(\nabla_{ki}u \nabla_{kj} u + Kv \nabla_{ij} u+  \nabla_k u \nabla_{kij} u) \Big).
			\end{split}
		\end{align}
		Thus, from 	$\left(W^2\, \delta_{ij}-\nabla_i u\nabla_j u\right)\nabla_{ij}\chi(x_0)\leq 0$, we get
		\begin{align}
			\label{CC8}
			\begin{split}
				\left((1+v)\delta_{ij} - \nabla_i u\nabla_j u\right)&\nabla_k u\nabla_{ijk}u \leq K^2 v^2  -Kv(1+v)^{3/2}e^u\Upsilon \\ &- Knv(1+v)
			\end{split}
		\end{align}
		Using the Ricci identity
		\[
		\nabla_{ijk} u = \nabla_{ikj} u + \nabla_l u R_{iljk}
		\]
		we get
		\begin{align}
			\label{CC82}
			\begin{split}
				\left((1+v)\delta_{ij} - \nabla_i u\nabla_j u\right)\nabla_k u\nabla_{ijk}u &=  (1+v)\nabla_\ell u\nabla_k u R_{i\ell jk} \\
				+	\left((1+v)\delta_{ij} - \nabla_i u\nabla_j u\right)&\nabla_k u\nabla_{kij}u 
			\end{split}
		\end{align}
		since $\nabla_\ell u\nabla_k u\nabla_i u\nabla_j u \delta_{ij}R_{i\ell jk}=0$. 	Replacing  into \eqref{CC8}, we obtain
		\begin{align}
			\label{CC9}
			\begin{split}
				\left((1+v)\delta_{ij} - \nabla_i u\nabla_j u\right)&\nabla_k u\nabla_{ijk}u \leq K^2 v^2  -Kv(1+v)^{3/2}e^u\Upsilon\\ & - Knv(1+v)+v(1+v)R
			\end{split}
		\end{align}
		where $R=R_{11}$.  Finally, substituting \eqref{CC9} into \eqref{CC6}, we get
		\begin{align}
			\begin{split}
				\label{CC11}
				K^2&(v^3-v^2)+Knv(1+v)^2+v(1+v)^2 R\\ &+K(1-2v-3v^2)ve^u\Upsilon +(1+v)^{5/2}v^{1/2}\nabla_1(e^u\Upsilon)\leq 0.
			\end{split}
		\end{align}
		As
		\[
		\nabla_1(e^u\Upsilon)=\frac{\partial}{\partial\rho}\left(\rho\Upsilon\right)\Big |_{\rho=e^u}e^u v^{1/2}+e^u\nabla_1 \Upsilon,
		\]
		we can apply  \eqref{decres-rho} to obtain from \eqref{CC11} the inequality
		\begin{align}
			\begin{split}
				\label{CC12}
				&	\left(K^2-3Ke^u\Upsilon+ R-e^u|\nabla_1\Upsilon|\right)v^3 \\ + & \left(2Kn+2R-2Ke^u\Upsilon-2e^u|\nabla_1\Upsilon|-K^2\right)v^2\\ 
				&	+\left(Ke^u\Upsilon-2e^u|\nabla_1\Upsilon|+Kn+R\right)v-e^u|\nabla_1\Upsilon|\leq 0.
			\end{split}
		\end{align}
		Notice that all constants in the expression above depend only on $|u|$ and $\Upsilon$. We can choose $K$ large enough so that the coefficient of $v^3$ in the left hand side is positive. With this choice we conclude that $v\leq C$ at $x_0$ for some uniform constant depending on the height and boundary gradient estimates for $u$ and on $\Upsilon$.
	\end{proof}
	
	\section{Existence}
	\label{Existence}
	
	In this section, we establish the proof of Theorem \ref{main} by applying the continuity method combined with a degree theory argument (see \cite{YY}) and utilizing the a priori estimates we have established in the previous sections. Our proof follow the lines of \cite{A-L-L} and \cite{FF-A}.  As noted in section \ref{Preliminaries}, to obtain a solution of \eqref{pd}, we work with the formulation given in \eqref{pd2}. Specifically, we will consider two auxiliary forms for \eqref{pd} obtained by choosing $\Upsilon$ as follows: For a fixed $t\in [0,1]$ we define the functions $\Psi_t$ and $\Theta_t$, defined on 
	\[
	\Delta=\left\{X\in\mathbb{R}^{n+1}: \frac{X}{|X|}\in\overline{\Omega}\right\}
	\]
	and given by
	\begin{equation}\label{Psi_t}
		\Psi_t(\rho x)=-\rho^{-C}\left(t\epsilon+(1-t)\overline{\rho}^C nH_{\overline{\Sigma}}\right)
	\end{equation}
	and
	\begin{equation}\label{theta_t}
		\Theta_t(\rho x)=-tnH-(1-t)\epsilon\rho^{-C},
	\end{equation}
	where $\overline{\rho}=e^{\overline{u}}, x\in \Omega$, and $ \epsilon>0$ satisfies
	\begin{equation}\label{ineq H}
		nH_{\overline{\Sigma}}>nH+\frac{\epsilon}{\ell^C},
	\end{equation}
	with $\ell=\inf e^{\overline{u}}$,	and $C>0$ is a constant to be chosen later. 
	These functions satisfy
	\begin{eqnarray*}
		\dfrac{\partial}{\partial \rho}(\rho \Psi_t(\rho x)) 
		=(C-1)\rho^{-C}\left(t\epsilon+(1-t)\overline{\rho}^CnH_{\overline{\Sigma}}\right)>0
	\end{eqnarray*}
	and
	\begin{eqnarray*}
		\dfrac{\partial}{\partial \rho}(\rho \Theta_t(\rho x))  
		&=&-tnH+(C-1)(1-t)\frac{\epsilon}{\rho^C}\\
		&>&-tnH + (C-1)(1-t)\epsilon\overline{\rho}^C >0
	\end{eqnarray*}
	for $\rho\leq \overline{\rho}$ and $C>0$ sufficiently large. We consider the two family of Dirichlet problem of the form (\ref{pd2}), by setting $\Upsilon=\Psi_t$ and $\Upsilon=\Theta_t$:
	\begin{equation}\label{pd3}
		\left\{\begin{array}{cl}
			Q[u]=-e^{-Cu}(t\epsilon+(1-t)ne^{C\overline{u}}H_{\overline{\Sigma}}) &\textrm{in } \Omega \\
			u= \varphi  &\textrm{on } \partial \Omega
		\end{array}\right.
	\end{equation}
	and
	\begin{equation}\label{pd4}
		\left\{\begin{array}{cl}
			Q[u]=-tnH-(1-t)\epsilon e^{-Cu}& \textrm{in } \Omega \\
			u=\varphi & \textrm{on } \partial \Omega
		\end{array}\right.
	\end{equation}
	for each $t\in[0,1]$. The zero order term of the linearized form of \eqref{pd3} satisfies
	\begin{equation}\label{30}
		\frac{\partial}{\partial u}\left(Q[u]-\Psi_t(e^ux)\right)=-\Psi_t(e^ux)+C\Psi_t(e^ux)<0,
	\end{equation}
	since, by \eqref{12},
	\begin{equation}
		\frac{\partial Q}{\partial u}=-Q[u] =-\Psi_t
	\end{equation}
	and
	\begin{equation}
		\frac{\partial}{\partial u}(\Psi_t(e^ux))=-\frac{\partial}{\partial u}\left(e^{-Cu}(t\epsilon+(1-t)nH_{\overline{\Sigma}})e^{C\overline{u}}\right)=-C\Psi_t(e^ux).
	\end{equation}
	
	Thus, \eqref{pd3} satisfies the Comparison Principle, and its linearized operator is invertible. Now we require the following lemmas.
	\begin{lemma}\label{lemma6}
		For each $t\in[0,1]$, \eqref{pd3}  has at most one solution $u$, and $u\leq \overline{u}$.
	\end{lemma}
	\begin{proof} Uniqueness, if a solution exists, follows from the Comparison Principle. We prove that any solution $u$ satisfies $u\leq \overline{u}$. Suppose, by contradiction, that $u-\overline{u}$ achieves a positive maximum at some interior point $x_0\in \Omega$.
		Thus,
		\begin{equation}\label{31}
			u(x_0)>\overline{u}(x_0), \qquad \nabla u(x_0)=\nabla\overline{u}(x_0),\qquad \nabla^2 u(x_0)\leq \nabla^2\overline{u}(x_0).
		\end{equation}
		Consider the deformation
		\begin{equation*}
			u_s:=su+(1-s)\overline{u},\qquad s\in [0,1]
		\end{equation*}
		and define $\beta(s)=Q[u_s]-\Psi_t(e^{u_s} x)$,
		where $t\in[0,1]$ is such that $u$ is a solution of  \eqref{pd3}.
		Since 
		\begin{eqnarray*}
			\beta(0) =Q[\overline{u}]-\Psi_t(e^{\overline{u}}x)   
			=t\left(\epsilon e^{-C\overline{u}}-nH_{\overline{\Sigma}}\right)\nonumber
		\end{eqnarray*}
		it follows from the choice of $\epsilon >0$ that 	$\beta(0)<0$.
		On the other hand, as $u$ is a solution of \eqref{pd3}, we have
		\[
		\beta(1)=Q[u]-\Psi_t(e^ux)=0.
		\]
		Thus, there exists $s_0\in [0,1]$ such that $\beta(s_0)=0$ and $\beta'(s_0)\geq 0$. Therefore, 
		\[
		Q[u_{s_0}(x_0)]=\Psi_t(u_{s_0}(x_0))
		\]
		and
		\begin{align*}
			\begin{split}
				\dfrac{\partial Q}{\partial \nabla_{ij} u}\bigg|_{u_{s_0}(x_0)}\nabla_{ij}(u-\overline{u})(x_0)+\dfrac{\partial Q}{\partial \nabla_{i} u}\bigg|_{u_{s_0}(x_0)}\nabla_{i}(u-\overline{u})(x_0)\\
				+\dfrac{\partial Q}{\partial u}\bigg|_{u_{s_0}(x_0)}(u-\overline{u})(x_0)+\dfrac{\partial \Psi_t}{\partial u}\bigg|_{u_{s_0}(x_0)}(u-\overline{u})(x_0)\geq 0.
			\end{split}
		\end{align*}
		Applying \eqref{31} to the above inequality, we obtain
		\begin{equation}\label{32}
			\dfrac{\partial Q}{\partial \nabla_{ij} u}\bigg|_{u_{s_0}(x_0)}\nabla_{ij}(u-\overline{u})(x_0)+\dfrac{\partial }{\partial u}(Q-\Psi_t)\bigg|_{u_{s_0}(x_0)}(u-\overline{u})(x_0)\geq 0.
		\end{equation}
	However,  from \eqref{30} and \eqref{31}, we have
		\[
		\dfrac{\partial Q}{\partial \nabla_{ij} u}\bigg|_{u_{s_0}(x_0)}\nabla_{ij}(u-\overline{u})(x_0)\leq 0\qquad \mbox{and}\qquad \dfrac{\partial }{\partial u}(Q+\Psi_t)\bigg|_{u_{s_0}(x_0)}(u-\overline{u})(x_0)< 0
		\]
		which contradicts \eqref{32}. Hence, $u\leq \overline{u}$.
	\end{proof}
	\begin{lemma}\label{lemma7}
		Let $u\leq \overline{u}$ be a solution of \eqref{pd3}. Then $u< \overline{u}$ in $\Omega$ and $\nu(u-\overline{u})<0$ on $\partial \Omega$, where $\nu$ is the inward unit normal to $\partial \Omega$.
	\end{lemma}
	\begin{proof} Suppose, by contradiction, that $u=\overline{u}$ at some point $x_0\in \Omega$. Then $x_0$ is a maximum point of $u-\overline{u}$. Thus, 
		\begin{equation}\label{33}
			u(x_0)=\overline{u}(x_0), \quad \nabla u(x_0)=\nabla\overline{u}(x_0),\quad \nabla^2 u(x_0)=\nabla^2\overline{u}(x_0).
		\end{equation}
		From \eqref{12}, it follows that
		\begin{equation}\label{34}
			Q[u](x_0)\leq Q[\overline{u}](x_0).
		\end{equation}
	However, by the choice of $\epsilon$, we have
		\begin{eqnarray*}
			Q[u](x_0)&=&-tnH-(1-t)\epsilon e^{-C\overline u(x_0)}  \\
			& >&-tnH- \epsilon \ell^{-C}>-nH_{\overline{\Sigma}}=Q[\overline u](x_0) 
		\end{eqnarray*}
		which contradicts \eqref{34}. Therefore, $u<\overline{u}$ in $\Omega$.
		
	Since $u\leq \overline{u}$, it follows that  $\nu (u-\overline{u})\leq 0$ on $\partial\Omega$. Suppose  $\nu (u-\overline{u})(x_0)=0$  for some point $x_0\in \partial \Omega$. Since $u=\overline{u}$ on $\partial \Omega$, we have $\nabla u(x_0)=\nabla \overline{u}(x_0)$. From the proof of the first part, we conclude that $\nabla^2 u(x_0)>\nabla^2 \overline{u}(x_0)$, as otherwise, we would obtain a contradiction as before. Thus, there exists $\xi\in T_{x_0}\mathbb S^n$ such that
		\[
		\nabla_{\xi \xi}u(x_0)>\nabla_{\xi \xi}\overline{u}(x_0).
		\]
		Moreover, $\xi$ cannot be tangent to $\partial \Omega$, since $u=\overline{u}$ on $\partial \Omega$ and $\nabla u(x_0)=\nabla \overline{u}(x_0)$.  Assume that $\xi$ (or $-\xi$) points inward to $\Omega$. Let $\gamma(t)=\exp_{x_0}(t\xi )$ be the geodesic starting at $x_0$ with velocity $\xi$, that is,  For sufficiently small $t$, $\gamma(t)\in \Omega$ and we have 
		\begin{align}
			\begin{split}
				\label{conta-exist}
				u(\gamma(0))&=\overline{u}(\gamma(0))\\
				(u\circ \gamma)'(0)&=(\overline{u}\circ \gamma)'(0)\\
				(u\circ \gamma)''(0)&>(\overline{u}\circ \gamma)''(0).  
			\end{split}
		\end{align}
		Thus, $u(\gamma(t))>\overline{u}(\gamma(t))$ for sufficiently small $t$, which  contradicts $u< \overline{u}$ in $\Omega$. Hence, $\nu(u-\overline{u})<0$ on $\partial\Omega$.
	\end{proof}
	
	We are now in a position to establish the existence of solution for \eqref{pd3} and \eqref{pd4}.
	
	\begin{theorem}\label{thm9}
		For each $t\in[0,1]$, \eqref{pd3} has a unique solution.
	\end{theorem}
	\begin{proof}
	Uniqueness follows from the ellipticity of $Q$ and \eqref{30}. To prove the existence, we use the standard continuity method. Define
		\[
		\mathcal{S}=\{t\in [0,1], (\ref{pd3})\quad\mbox{has solution} \}.
		\]
		We have $ 0\in\mathcal{S}\neq\emptyset$ since $\overline{u}$ is a solution of  \eqref{pd2} for $t=0$. Moreover, since \eqref{30} implies that the linearized operator of \eqref{pd3} is invertible, $\mathcal{S}$ is  open.  Finally, $\mathcal{S}$ is closed,  as Lemma \ref{lemma6} ensures that each solution $u$ of \eqref{pd3} satisfies $u\leq \overline{u}$, and since $\frac{\partial }{\partial \rho}\left(\Psi(\rho x)\right)\geq 0$, the a priori estimates established in  previous sections apply. Thus, $\mathcal{S}=[0,1]$.
	\end{proof}
	
	\begin{theorem}
		For each $t\in [0,1]$, \eqref{pd4} has a solution.
	\end{theorem}
	\begin{proof}
		First, observe that if $u$ is a solution of \eqref{pd4} with $u\leq \overline{u}$, then
		\begin{equation}\label{35}
			\frac{\partial}{\partial \rho}(\rho \Theta_t(\rho x))>0,
		\end{equation}
		holds for $C$ sufficiently large. Thus,  the a priori estimates from previous sections apply. By the standard regularity  theory for quasilinear second-order elliptic equations, higher-order estimates also hold.
		In particular, there exists a uniform constant $C_1>0$ such that
		\begin{equation}\label{36}
			|u|_{C^{4,\alpha}(\overline{\Omega})}<C_1.
		\end{equation}
		Let $\mathcal{O}$ be the subset of $C^{4,\alpha}(\overline{\Omega})$ defined by
		\begin{align}
			\begin{split}
				\mathcal{O}=\left\{ \right. v\in C_0^{4,\alpha} (\overline{\Omega})  &: v>0\ \mbox{in }\Omega,\ \nabla_\nu v>0\ \mbox{in }\partial \Omega \\  \mbox{and} &\ |v|_{4,\alpha}<C_1+|\overline{u}|_{4,\alpha}\left.\right\},
			\end{split}
		\end{align}
		where $C_1$ is the constant from \eqref{36}.  The set $\mathcal{O}$ is a bounded open subset of $C_0^{4,\alpha}$. Consider the map $M:\mathcal{O}\times [0,1]\longrightarrow C^{2,\alpha}(\Omega)$ defined by
		\[
		M_t[v]=M[v,t]=Q[\overline{u}-v]-\Theta_t(e^{\overline{u}-v}x).
		\]
		By Theorem \ref{thm9}, there exists a unique solution $u^0$ of \eqref{pd3} for $t=0$. 
		Thus, $v^0:=\overline{u}-u^0$ is a solution of $M_0[v]=0$. Furthermore, by Lemma \ref{lemma6}, $v^0\geq 0$, and by Lemma \ref{lemma7}, $v^0>0$ in $\Omega$ and $\nabla_\nu v^0> 0$ on $\partial \Omega$. Since $|v^0|_{4,\alpha}\leq C_1+|\overline{u}|_{4,\alpha}$, we have $v^0\in \mathcal{O}$.
		
		The equation  $M_t[v]=0$ has no solution on $\partial \mathcal{O}$. Indeed, if $v$ with $C^{4,\alpha}$ norm $C_1+|\overline{u}|_{4,\alpha}$ is a solution, it contradicts \eqref{36}. If $v=0$ at some interior point or  $\nabla_{\nu}v=0$ at some boundary point, it contradicts Lemma \ref{lemma7}. Moreover, $M_t$ is uniformly elliptic in $\mathcal{O}$, independent of $t$. Thus, the degree $\deg(M_t,\mathcal{O},0)$ is well-defined and independent of $t$.
		
		We compute $\deg(M_0,\mathcal{O},0)$. Since $v^0$ is the unique solution of $M_0[v]=0$ in $\mathcal{O}$, the Fréchet derivative of $M_0$ at $v^0$ is the linear operator from $C_0^{4,\alpha}(\Omega)$ to $C^{2,\alpha}(\Omega)$ given by
		\[
		M_{0,v^{0}}(h)=\frac{\partial Q}{\partial \nabla_{ij}u }\bigg|_{v^0}\nabla_{ij}h+\frac{\partial Q}{\partial \nabla_{i}u }\bigg|_{v^0}\nabla_{i}h+\frac{\partial }{\partial u }(Q-\Theta_0)\bigg|_{v^0} h.
		\]
		Since
		\begin{equation}
			\frac{\partial}{\partial u}(Q-\Theta)\bigg|_{v^0}h=-Q[u^0]-\epsilon e^{-Cu^0}=-2\epsilon e^{-Cu^0}<0,
		\end{equation}
		$M_{0,v^0}$ is invertible. Hence, by the theory developed in \cite{YY},
		\begin{equation}
			\deg (M_0,\mathcal{O},0)=\deg (M_0,v^0,B_1,0)=\pm 1\neq 0,
		\end{equation}
		where $B_1$ is the unit ball of $C^{4,\alpha}_0(\overline\Omega)$.  Thus,
		\begin{equation}
			\deg(M_t,\mathcal{O},0)\neq 0,
		\end{equation}
		for all $t\in[0,1]$. In particular,  \eqref{pd4} has a solution for each $t\in [0,1]$.
		Let $v^1$ be the solution of $M_1[v]=0$. Then $u^1:=\overline{u}-v^1$ is a solution of \eqref{pd}, proving Theorem \ref{main}.
	\end{proof}

\end{document}